\documentclass[10pt,oneside,a4paper,draft]{article} 	

\usepackage[fleqn]{amsmath} 						
\usepackage{amsthm,amsfonts,latexsym,amssymb,amscd} 	
\usepackage[mathscr]{eucal}   						
\usepackage[all,2cell]{xy} \UseAllTwocells 		
\usepackage{color}
\usepackage{txfonts}

\usepackage[draft=false]{hyperref} 							

\renewcommand{\H}{\mathcal{H}} 	
 
\newcommand{\CC}{{\mathbb{C}}}

\newcommand{\As}{{\mathscr{A}}}\newcommand{\Bs}{{\mathscr{B}}}\newcommand{\Cs}{{\mathscr{C}}}

\newcommand{\Es}{{\mathscr{E}}}\newcommand{\Fs}{{\mathscr{F}}}\newcommand{\Gs}{{\mathscr{G}}}

\newcommand{\Ms}{{\mathscr{M}}}\newcommand{\Ns}{{\mathscr{N}}}
\newcommand{\Rs}{{\mathscr{R}}}\newcommand{\Ss}{{\mathscr{S}}}

\DeclareFontFamily{U}{rsfs}{\skewchar\font127 }
\DeclareFontShape{U}{rsfs}{m}{n}{%
   <5> <6> rsfs5
   <7> rsfs7
   <8> <9> <10> <10.95> <12> <14.4> <17.28> <20.74> <24.88> rsfs10
}{}
\DeclareSymbolFont{rsfs}{U}{rsfs}{m}{n} 
\DeclareSymbolFontAlphabet{\scr}{rsfs}

\newcommand{\Af}{\scr{A}}\newcommand{\Ef}{\scr{E}}

\newcommand{\Mf}{\scr{M}} 
\newcommand{\Sf}{\scr{S}}\newcommand{\Tf}{\scr{T}}

\DeclareMathOperator{\CCl}{\CC l}

\renewcommand{\emph}{\textbf} 										
\newcommand{\ip}[2]{\langle #1\mid #2\rangle}					

\newcommand{\hlink}[2]{\href{#1}{\texttt{#2}}} 

\newcommand{\xqedhere}[2]{%
  \rlap{\hbox to#1{\hfil\llap{\ensuremath{#2}}}}}

\theoremstyle{plain}
\newtheorem{theorem}{Theorem}[section]							

\newtheorem{proposition}[theorem]{Proposition}

\theoremstyle{definition}

\numberwithin{equation}{section}  		
\title{\textbf{Remarks on Morphisms of Spectral Geometries\footnote{
This is a revised version, only for arXiv, of a paper published in East-West Journal of Mathematics (2013) 15(1):15-24. 
 \newline 
A previous draft of this work appeared in the collection of contributed full papers of the ICMA-MU 2013 ``International Conference in Mathematics and Applications - Mahidol University'' (organized in Bangkok on  19-21 January 2013) where, for personal reasons, the first author refused to have his name mentioned. 
}}}

\author{
\normalsize Paolo Bertozzini\footnote{
Supported by Thammasat University Research Scholar Grant n.2/15/2556: ``Categorical Non-commutative Geometry''.
} 
\\
\normalsize \textit{Department of Mathematics and Statistics, Faculty of Science and Technology,}
\\
\normalsize \textit{Thammasat University, Bangkok 12121, Thailand}
\\
\normalsize e-mail: \texttt{paolo.th@gmail.com}
\\
\\
\normalsize Fred Jaffrennou\footnote{Partially supported by an RA scholarship, academic year 2012, from the Faculty of Graduate Studies, Mahidol University.} 
\\
\normalsize \textit{Department of Mathematics, Faculty of Science} 
\\
\normalsize \textit{Mahidol University, Bangkok 10400 Thailand} 
\\
\normalsize email: \texttt{fredjaffr@yahoo.fr}
}

\date{\normalsize{published version: 26 May 2013 
\quad 
revised version: 08 April 2014 
}}

\begin{document}

\maketitle

\begin{abstract}
Non-commutative geometry, conceived by Alain Connes, is a new branch of mathematics whose aim is the study of geometrical spaces using tools from operator algebras and functional analysis. Specifically metrics for non-commutative manifolds are now encoded via spectral triples, a set of data involving a Hilbert space, an algebra of operators acting on it and an unbounded self-adjoint operator, maybe endowed with supplemental structures.

Our main objective is to prove a version of Gel'fand-Na\u\i mark duality adapted to the context of Alain Connes' spectral triples.

In this preliminary exposition, we present:
\begin{itemize}
\item
a description of the relevant categories of geometrical spaces, namely compact Hausdorff smooth finite-dimensional orientable Riemannian manifolds, or more generally Hermitian bundles of Clifford modules over them;
\item
some tentative definitions of categories of algebraic structures, namely commutative Riemannian spectral triples;
\item
a construction of functors that associate a naive morphism of spectral triples to every smooth (totally geodesic) map.
\end{itemize}
The full construction of spectrum functors (reconstruction theorem for morphisms) and a proof of duality between the previous ``geometrical'' and ``algebraic'' categories are postponed to subsequent works, but we provide here some hints in this direction.

We also conjecture how the previous ``algebraic'' categories might provide a suitable environment for the description of morphisms in non-commutative geometry.

\smallskip


\smallskip

\noindent
\emph{MSC-2010:}
					46L87,			
					46M15, 			
					46L08,			
					46M20, 			
					16D90.			
\end{abstract}




\section{Introduction}

In A.Connes' non-commutative geometry~\cite{C,FGV,Lan}, every compact Hausdorff spinorial Riemannian finite-dimensional orientable manifold
$(M,g_M)$ with a spinorial Hermitian bundle $S(M)$ and volume form $\mu_{g_M}$ is associated to a commutative regular spectral triple $(\As,\H,D_M)$ where: $\As:=C(M;\CC)$ is the unital commutative C*-algebra of complex valued continuous functions on $M$ with respect to the maximum modulus norm; $\H:=L^2(S(M))$ is the Hilbert space obtained by completion of the \hbox{$C(M)$-module} $\Gamma(S(M))$ of continuous sections of the spinor bundle with respect to the norm induced by the inner product \hbox{$\ip{\sigma}{\rho}:=\int \ip{\sigma_x}{\rho_x}_{S(M)_x}\, \text{d}\mu_{g_M}$}, for all $\sigma,\rho\in \Gamma(S(M))$; and $D_M$ is the Dirac operator i.e.~the closure of the densely defined essentially self-adjoint operator obtained by contracting the spinorial Levi-Civita connection with the Clifford multiplication.

A reconstruction theorem proved by A.Connes~\cite{C11,C12} (see also~\cite{Re1,RV} for previous only partially successful attempts) assures that a commutative spectral triple (that is irreducible real, graded, strongly regular $m$-dimensional finite absolutely continuous orientable with totally antisymmetric Hochschild cycle in the last $m$ entries, and satisfying Poincar\'e duality) is naturally isomorphic to the above mentioned canonical spectral triple of a spinorial Riemannian manifold with a given Hermitian spinor bundle equipped with charge conjugation.
The reconstruction theorem has been recently extended to cover the case of Riemannian spectral triples~\cite{LRV} and to more general situations of almost commutative (real) spectral triples~\cite{Ca,Ca2}.

It is still an open problem to reformulate these reconstruction theorems for (almost) commutative spectral triples in a fully categorical context in the same spirit of such celebrated cornerstones of non-commutative topology as Gel'fand-Na\u\i mark duality (between categories of continuous maps of compact Hausdorff topological spaces and categories of commutative unital $*$-homomorphisms of unital C*-algebras), Serre-Swan equivalence (between vector bundles and finite projective modules), or Takahashi duality (between Hilbert bundles over compact Hausdorff spaces and Hilbert 
C*-modules over commutative unital C*-algebras).

As a first step towards such duality results, several suggestions for the construction of categories of spectral triples have been put forward (see for example~\cite{CM6,B1,B3,B4,B5,B6} and the references therein). Of particular relevance is the category of spectral triples recently constructed by
B.Mesland~\cite{Me,Me2}, where morphisms are Kasparov KK-bimodules equipped with smooth structure and connection.

In this very preliminary and tentative account our purpose, in the spirit of Cartesian geometry, is to suggest a description of some possible dualities between categories of geometrical spaces (usually compact Hausdorff smooth finite-dimensional orientable Riemannian manifolds or more generally Hermitian bundles of Clifford modules over them), here collectively denoted by $\Tf$, and categories of algebraic functional analytic structures (usually some variants of Connes' spectral triples) here denoted by $\Af$.
The dualities are realized via two contravariant functors, the section functor \hbox{$\Gamma:\Tf\to\Af$} and the spectrum functor $\Sigma:\Af\to\Tf$ as in the following diagram:
\begin{equation*}
\xymatrix{
\Tf \rtwocell^\Gamma_\Sigma{'}  & \Af.
}
\end{equation*}
In the commutative C*-algebras context, we will describe how to embed categories of smooth (totally geodesic) maps of compact Riemannian manifolds into more general categories of Hermitian bundles and we will also see how a section functor can be used to trade such categories of bundles with categories of Hilbert C*-bimodules.

In the non-commutative C*-algebra case, we will mainly deal with topological situations, discussing only the rather special categories of 
``factorizable'' 
Hilbert \hbox{C*-bi}\-modules over tensor products of unital C*-algebras over commutative subalgebras. A more complete study aiming at the construction of functors from Riemann manifolds to B.Mesland's category of spectral triples and to the possible definition of involutive categories of spectral triples is left for future work.

\section{Categories of Manifolds, Bundles and Propagators}

The objects of our categories will be, for now, compact Hausdorff smooth Riemannian orientable finite-dimensional manifolds that are not necessarily connected.\footnote{For background on manifolds, bundles and differential geometry, the reader is referred for example to  R.Abraham, J.Marsden, T.Ratiu~\cite{AMR} and L.Nicolaescu~\cite{N}.} We have several interesting categories that can be naturally constructed:
\begin{itemize}
\item[a)]
The category $\Mf^\infty$ of smooth maps between such manifolds and its subcategories $\Mf^\infty_e$ of smooth embeddings\footnote{Here and in all the subsequent items we could also consider categories of (injective) immersions in place of embeddings.} and $\Mf^\infty_s$ of smooth submersions.
\item[b)]
The category R-$\Mf^\infty_e$ of smooth maps that are Riemannian embeddings and R-$\Mf^\infty_s$ of smooth Riemannian submersions.
\item[c)]
The category R-$\Mf^\infty_{ge}$ of totally geodesic smooth Riemannian embeddings and R-$\Mf^\infty_{gs}$ of totally geodesic smooth Riemannian submersions.
\item[d)]
The category R-$\Mf^\infty_{gec}$ of totally geodesic smooth Riemann embeddings of connected components and R-$\Mf^\infty_{gsc}$ of totally geodesic smooth Riemannian coverings.
\end{itemize}
There are natural inclusion functors between such categories as in the following diagrams:
\begin{gather*}
\xymatrix{
\text{R-}\Mf^\infty_{gec} \ar@{^(->}[r] & \text{R-}\Mf^\infty_{ge} \ar@{^(->}[r] & \text{R-}\Mf^\infty_{e}  \ar@{^(->}[r] & \Mf^\infty_{e}
\ar@{^(->}[r] &  \Mf^\infty}
\\
\xymatrix{\text{R-}\Mf^\infty_{gsc} \ar@{^(->}[r] & \text{R-}\Mf^\infty_{gs}  \ar@{^(->}[r] & \text{R-}\Mf^\infty_{s}  \ar@{^(->}[r] & \Mf^\infty_{s}
\ar@{^(->}[r] &  \Mf^\infty
}
\end{gather*}
The previous categories are not equipped with involutions, since the reciprocal relations are generally not functions, furthermore the categories of embeddings and submersions appear in a kind of dual role. A more satisfactory involutive environment can be obtained considering  (in the terminology often used in algebraic geometry) cycles i.e.~relations $R$ between such manifolds that are themselves compact (respectively (totally geodesic) Riemannian) orientable sub-manifolds of the product manifold $M\xleftarrow{\pi_M}M\times N\xrightarrow{\pi_N}N$ and equipping them with ``bundle-propagators'' between the tangent bundles $T(M)$ and $T(N)$ i.e.~smooth Hermitian sub-bundles of $\pi_M^\bullet(T(M))\oplus\pi_N^\bullet(T(N))|_R$ that are fiberwise linear (partial isometric, or equivalently partial co-isometric) relations between the corresponding fibers of the pull-backs on $R$ of the tangent bundles of $M$ and $N$.\footnote{
Since the equalizer of smooth maps between smooth manifolds usually is not a smooth manifold, strictly speaking, the composition of smooth ((totally geodesic) Riemannian) cycles between Hausdorff compact Riemannian orientable manifolds fails to be another such manifold. In order to solve this problem it is appropriate to embed the previous categories of manifolds into the corresponding categories of Hausdorff compact Riemannian orientable finite-dimensional \emph{diffeological spaces}~\cite{I-Z,La} and from now on, whenever necessary, we will assume that such embedding has been done.} 
Furthermore, in order assure the closeness under composition of this category of bundle-propagators, we will actually work with smooth ((totally geodesic) Riemannian) \emph{relational spans} $M\xleftarrow{\rho_M}R \xrightarrow{\rho_N}N$ of such compact Hausdorff Riemannian manifolds (or diffeological spaces).

More generally, we can further ``decouple'' the Hermitian bundles from the underlying Riemannian structure of the manifolds allowing ``(amplified) propagators'' between arbitrary Hermitian bundles of Clifford modules over the given manifolds that are equipped with a compatible connection.
In more detail, given two smooth (diffeological) Hermitian bundles $(E^1,\pi^1,X^1)$ and $(E^2,\pi^2,X^2)$ over compact Hausdorff smooth orientable finite-dimensional Riemannian manifolds (diffeological spaces) $X^1$ and $X^2$, here is a description of the morphisms in some of the several relevant categories $\Ef$ of bundles:
\begin{itemize}
\item[$\Ef^1$]
The usual categories of bundle morphisms: $(f,F)$ where $f:X^1\to X^2$ is a morphism of manifolds (diffeological spaces) in any of the previous categories $\Mf$ and $F: E^1\to E^2$ is a smooth map such that $\pi^2\circ F=f\circ\pi^1$ and that is respectively fiberwise linear, isometric (when $f$ is in $\Mf_e$), co-isometric (if $f$ is in $\Mf_s$).
\item[$\Ef^2$]
The category of Takahashi bundle morphisms~\cite{Ta2}: $(f,F)$ where the map $f:X^1\to X^2$ is as above and $F: f^\bullet(E^2)\to E^1$ is a morphism of bundles over $X^1$ in the previous sense.
\item[$\Ef^3$]
The category of \emph{propagators of bundles}: $(E,\gamma,R)$ where $R$ is a smooth ((totally geodesic) Riemannian) relational span 
$X^1\xleftarrow{\rho_1}R \xrightarrow{\rho_2} X^2$ and $E$ is the total space of an Hermitian sub-bundle, over $R$, 
of the Whitney sum $\rho_1^\bullet(E^1)\oplus \rho_2^\bullet(E^2)$, that is a fiberwise partial isometry i.e.~the fiber $E_{r}:=\gamma^{-1}(r)\subset \rho_1^\bullet(E^1)_r\oplus \rho_2^\bullet(E^2)_r$ is the graph of a partial isometry between $\rho^\bullet_1(E^1)_r$ and $\rho^\bullet_2(E^2)_r$, for all $r\in R$.\footnote{This category, as well as the category $\Ef^4$, is involutive and its morphisms can be considered as a bivariant version of Takahashi bundle morphisms.}
\item[$\Ef^4$]
The category of amplified propagators of bundles: $(E,\gamma,R)$, where $R$ is a relational span as above and $E$ is an Hermitian sub-bundle of the Whitney sum $\rho_1^\bullet(E^1\otimes W^1)\oplus \rho_2^\bullet(E^2\otimes W^2)$, for two given Hermitian bundles $W^1$ over $X^1$ and 
$W^2$ over $X^2$ in such a way that, for every $r\in R$, $E_r$ is the graph of a partial isometry.\footnote{ 
Here, and in the category $\Ef^3$, the Hermitian structure on $E_r$, $r\in R$ is uniquely determined by the isometry requirement for the projections and it is a rescaling of the metric induced by the orthogonal Whitney sum. 
\\ 
More generally one can simply consider spans of fiberwise isometries 
$\rho_1^\bullet(E^1)\xleftarrow{\pi_1}E \xrightarrow{\pi_2} \rho_2^\bullet(E^2)$ of Hermitian bundles over $R$. 
}
\end{itemize}
We have natural inclusions relating the previous categories as follows:
\begin{equation*}
\xymatrix{
\Ef^1 \ar@{^(->}[r] & \Ef^3 & & \Ef^2 \ar@{^(->}[r] & \Ef^3 & & \Ef^3 \ar@{^(->}[r] & \Ef^4.
}
\end{equation*}
Whenever we have bundles of Clifford modules that are equipped with Clifford connections, we can require our morphisms to be stable under the action of the tensor product of the Clifford bundles and totally geodesic for the connection.

Exploiting the language of 2-categories, we can produce an even more efficient way to encode such categorical structures:\footnote{For details on higher categories, the reader is referred for example to T.Leinster~\cite{L}.} objects are compact Hausdorff smooth orientable finite-dimensional manifolds (diffeological spaces) $X^1,X^2$; 1-arrows are Hermitian bundles $E^1,E^2$ (eventually equipped with a Clifford action and a compatible connection) over relational spans between $X^1$ and $X^2$; 2-arrows are (amplified) propagators between such 1-arrows bundles, that can be required to be stable under the Clifford action and totally geodesic for the connection. Note that, since 2-arrows are themselves bundles over relational spans, the construction of arrows can be iterated obtaining arbitrary higher categories $\Tf$ of bivariant bundles over relational spans.

We can now sketch the construction of embedding functors from the several categories $\Mf$ of manifolds to $\Ef$ of bundles (and so into the higher categories $\Tf$ of bivariant bundles).
\begin{theorem}
We have covariant Grassmann functors $\Lambda^\CC: \Mf\to\Ef$ from the previous categories of manifolds into the category $\Ef$
of (amplified) propagators of bundles.
\end{theorem}
\begin{proof}
On the objects, the functor $\Lambda^\CC$ associates to every smooth orientable Riemannian manifold $M$ its complexified Grassmann algebra Hermitian bundle $\Lambda^\CC(M)$ with its natural right and left Clifford actions of the complexified Clifford algebra bundle $\CCl(M)$ and with the induced Levi-Civita Ehresman connection.

On the arrows, the functor $\Lambda^\CC$ associates to every smooth map $f:X^1\to X^2$ the complexified Bogoljubov second quantized
$\Lambda^\CC(Df):\Lambda^\CC(X^1)\to\Lambda^\CC(X^2)$ of the differential map \hbox{$Df:T(X^1)\to T(X^2)$} of $f$. If the map $f$ is a (totally geodesic) Riemannian isometry or co-isometry, fiberwise the graph of $\Lambda^\CC(Df)$ is an isometry or co-isometry and hence determines a propagator bundle. The complexified Clifford  functor $\CCl$ associates to every object $M$ its complexified Clifford bundle $\CCl(M)$ and to every Riemannian (co)isometry $f$ an amplified propagator of the Clifford bundles that induces a right/left Clifford action on the propagator bundle determined by $\Lambda^\CC(Df)$ between the Grassmann bundles. For totally geodesic maps, the covariant derivative on the Whitney sum of the Grassmann bundles decomposes inducing a covariant derivative on the propagator bundle.
\end{proof}

If we examine in more detail how totally geodesic maps between compact Riemannian manifolds are described in terms of propagators, we see that the isometric differential map \hbox{$Df:T(M)\to T(N)$} induces an orthogonal splitting $T(N)|_{f(M)}=Df(T(M))\oplus Df(T(M))^\perp$ of the restriction to $f(M)$ of the tangent bundle of $N$. 
Passing to the complexified Grassmann bundles, and similarly for the Clifford bundles, we obtain the following tensorial decompositions
\begin{gather*}
\Lambda^\CC(T(N)|_{f(M)})\simeq\Lambda^\CC(Df(T(M)))\otimes\Lambda^\CC(Df(T(M))^\perp),
\\
\CCl(T(N)|_{f(M)})\simeq\CCl(Df(T(M)))\otimes\CCl (Df(T(M))^\perp).
\end{gather*}
For totally geodesic maps, the restriction of the Levi-Civita connection on $T(N)|_{f(M)}$ decomposes as a direct sum of the connections on the subbundles $Df(T(M))$ and $Df(T(M))^\perp$ and, denoting by $\nabla^N$, $\nabla^M$ and $\nabla^\perp$ the connection induced respectively on $\Lambda^\CC(T(N)|_{f(M)})$, $\Lambda^\CC(Df(T(M)))$ and $\Lambda^\CC(Df(T(M))^\perp)$, we have $\nabla^N=(\nabla^M\otimes I) \oplus (I\otimes \nabla^\perp)$ and contracting with the Clifford actions we obtain the following relation $D_N|_{f(M)}=(D_M\otimes I)\oplus (I\otimes D_\perp)$ between the Hodge-De Rham Dirac operators for $M$ and $N$, where $D_\perp$ denotes a ``transversal'' operator obtained contracting the Clifford action with the orthogonal part of the connection
$\nabla^\perp$. The interesting part, in view of the future study of links with the notion of B.Mesland morphisms of spectral triples, is the fact that the  Grassmann bundle $\Lambda^\CC(N)$ decomposes as a tensor product of a ``copy'' of the Grassmann bundle of $M$ with a ``transversal'' factor that, passing to the module of sections, will provide a Mesland morphism between the Hodge-De Rham spectral triples of $M$ and $N$.

\section{Naive Categories of Spectral Geometries}

In this section we try to examine some very tentative candidates for categories $\Af$ of non-commutative spectral geometries that might be used as targets for  functors that are defined on the categories $\Ef$ of bundles described in the previous section. Our general ideology will be to start at the topological level from Takahashi duality~\cite{Ta2} (that generalizes the well-known Gel'fand-Na\u\i mark duality between compact Hausdorff spaces and unital commutative C*-algebras) and proceed from there progressively adding the additional structures (Clifford actions, connections) that are required for the description of more rigid geometrical settings. Since Takahashi duality is between Hilbert bundles over compact Hausdorff spaces and Hilbert C*-modules over commutative unital C*-algebras, it is natural for us to start working on Hilbert C*-(bi)modules rather than on Hilbert spaces.
This explains our need to partially reformulate a naive notion of A.Connes spectral triples in the case of Hilbert C*-modules.

For our purpose here, a (naive) \emph{spectral triple} $(\As,\H,D)$ is given by a (possibly non-commutative) unital C*-algebra $\As$ faithfully represented on the Hilbert space $\H$ and a (possibly unbounded) self-adjoint operator $D$ with compact resolvent and such that the commutator $[D,x]$ extends to a bounded operator on $\H$, for all $x$ in a dense unital \hbox{C*-subalgebra} of $\As$ leaving invariant the domain of $D$.
We will reserve the terms \emph{Atiyah-Singer spectral triples} and \emph{Hodge-De Rham spectral triples} for all those spectral triples, with commutative C*-algebras $\As$, for which respectively either A.Connes' or S.Lord-A.Rennie-J.Varilly's reconstruction theorems~\cite{C11,LRV} are viable.

We say that $(\As,\Ms,D)$ is a naive left \emph{spectral module triple} if $\Ms$ is a unital left Hilbert C*-module, over the unital C*-algebra $\As$, that is equipped with a (possibly unbounded) regular operator $D$ such that, for all $x$ in a dense unital C*-subalgebra of $\As$ leaving invariant the domain of $D$, the commutator $[D,x]$ extends to an adjointable operator on $\Ms$.

The first category of spectral geometries that we consider is strictly adapted to the commutative algebra situation and will be in duality with the categories of (amplified) propagators already described.\footnote{
For the case of finitely generated projective Hilbert C*-modules. 
}
\begin{proposition}
There is an involutive  category $\Af^1$ of propagators of unital Hilbert \hbox{C*-modules} over commutative unital C*-algebras whose morphisms from the module $\Ms_\As$ to the module $\Ns_\Bs$ are given by Hilbert C*-modules 
$\Es_\Rs$ that are graphs $\Es_\Rs\subset(\Rs\otimes_\As\Ms)\oplus(\Rs\otimes_\Bs\Ns)$ of isometric morphisms of Hilbert C*-modules on $\Rs$, where
$\Rs$ is a unital C*-algebra bimodule over $\As\otimes_\CC\Bs$.\footnote{
More generally we can consider spans of isometries of Hilbert C*-modules 
$\Rs\otimes_\As\Ms\xleftarrow{\Lambda_1}\Es_\Rs \xrightarrow{\Lambda_2} \Rs\otimes_\Bs\Ns$ over $\Rs$.
}

\end{proposition}
The details of the proposition can be obtained considering that the section functor $\Gamma$ from Hilbert bundles to Hilbert C*-modules preserves direct sums and transforms the pull-back of bundles into change of the base algebra of modules via tensor product. In such commutative setting, if necessary, further requirements can be added to assure that these propagators of bimodules correspond to (totally geodesic) Riemannian maps.

Note that, as always, propagators consist of two distinct processes: first a \emph{transport} via pull-back of bundles and Hilbert C*-modules onto a common space here realized via the change of rings with tensorization over $\As$ and $\Bs$ and then a \emph{correspondence} here realized via the selection of suitable submodules in the direct sum.

For the special case of spectral module triples $(\As,\Ms_1,D_1)$ and $(\As,\Ms_2,D_2)$ on the same algebra $\As$, we can further specialize the propagators morphism of Hilbert C*-modules obtaining the following interesting definition of a category of spectral correspondences.

\begin{proposition}
There is a \emph{naive totally geodesic category of spectral correspondences module triples} $\Sf$ whose objects are naive spectral module triples over the same unital C*-algebra and whose morphisms, say from $(\As,\Ms_1,D_1)$ to $(\As,\Ms_2,D_2)$, consist of spectral module triples $(\As,\Phi,D_\Phi)$ where
$\Phi\subset \Ms_1\oplus \Ms_2$ is a left $\As$-submodule that is stable under the action of the regular operator $D_1\oplus D_2$ and $D_\Phi:=(D_1\oplus D_2)|_\Phi$.
\end{proposition}

The category $\Sf$ is essentially a bivariant version of the naive category of spectral triples~\cite{B1,B3,B6} and (at least in the commutative C*-algebra case) can be used to model the ``correspondence'' part in the definition of a propagator.

The ``transport'' process that in the commutative case is just a relatively unproblematic pull-back, in the case of non-commutative C*-algebras must be substituted by the more sophisticated notion of A.Connes' transfer of spectral triples  between different algebras via tensorization with appropriate bimodules (a process that has been further developed by B.Mesland).

Anyway, also the category $\Af^1$ is just an involutive version of the familiar category of Hilbert \hbox{C*-modules} over commutative unital C*-algebras used in Takahashi duality, where 1-arrows between C*-algebras reduce to unital \hbox{$*$-homo}morphisms. It is a general ideological principle that in non-commutative geometry categories of homomorphisms of algebras get substituted with categories of bimodules: every unital homomorphism $\phi:\As\to\Bs$ of unital \hbox{C*-alge}\-bras is associated to a pair of \emph{correspondences}: Hilbert C*-bimodules $\Bs_\As$ and $_\As\Bs$ (where the action of $\As$ on the right/left is via the homomorphism $\phi$) with $\Bs$-valued inner products. Composition of unital $*$-homomorphisms becomes the internal tensor product of such bimodules. As a consequence of this general passage from Abelian categories of bimodules to ``tensorial'' categories of bimodules, instead of pursuing the description of the details of dualities targeting the category $\Af^1$, it is important to try to look for a similar ``tensorial'' reformulation of the previous category.

A bivariant version of naive spectral triple is also needed and it is natural to start with a notion of Hilbert C*-bimodule. Although we are not ready yet to select a definition of Hilbert C*-bimodules over general non-commutative \hbox{C*-algebras}, we can provide some elementary examples of situations that are sufficient to cover at least some significant cases of Hilbert C*-bimodules over commutative C*-algebras. This will be enough to create an environment suitable for the formulation of dualities with subcategories of the previous categories of bundles that is more in line with generalizations to the non-commutative setting.
For this purpose, we define a \emph{unital C*-algebra bimodule, factorizable over commutative \hbox{C*-algebras}}, to be a unital bimodule ${}_\As\Rs_\Bs$ over the unital C*-algebras $\As$ and $\Bs$, such that $\Rs$ is a unital C*-algebra that is tensor product, over commutative unital \hbox{C*-algebras}, of other unital C*-algebra bimodules, i.e.~a unital C*-algebra of the form $\As\otimes_{C(Y)}\Fs\otimes_{C(X)}\Bs$, where 
$\As_{C(Y)},{}_{C(Y)}\Fs_{C(X)},{}_{C(X)}\Bs$ are three unital C*-algebra bimodules and $X,Y$ are compact Hausdorff spaces.\footnote{Note that, since the right/left actions of $\As$ and $\Bs$ on $\Rs=\As\otimes_{C(Y)}\Fs\otimes_{C(X)}\Bs$ commute, the C*-algebra $\Rs$ can be naturally considered as a bimodule over the unital \hbox{C*-algebras} $\As$ and $\Bs$, both on the right and on the left.}
A \emph{Hilbert C*-bimodule over a C*-algebra bimodule factorizable over commutative C*-algebras} is a unital bimodule ${}_\Rs\Ms_\Rs$ on a unital C*-algebra 
bimodule factorizable over commutative \hbox{C*-algebras} ${}_{\As}\Rs_\Bs=\As\otimes_{C(Y)}\Fs\otimes_{C(X)}\Bs$, 
that is of the form ${}_\Rs\Ms_\Rs=\As\otimes_{C(Y)}\otimes\widehat{\Ms}\otimes_{C(X)}\Bs$ where $\widehat{\Ms}$ is a bimodule over $\Fs$ that is also equipped with both right $\ip{\cdot}{\cdot}_\Fs$ and left ${}_\Fs\ip{\cdot}{\cdot}$ $\Fs$-valued inner products\footnote{Here both inner products are assumed to be Hermitian positive non-degenerate with the left product being 
left $\Fs$-linear: ${}_\Fs\ip{fx}{y}=f\cdot{}_\Fs\ip{x}{y}$ and right $\Fs$-adjointable: ${}_\Fs\ip{xf}{y}={}_\Fs\ip{x}{yf^*}$; and the  right product being right $\Fs$-linear: $\ip{x}{yf}_\Fs=\ip{x}{y}_\Fs \cdot f$ and left $\Fs$-adjointable: $\ip{fx}{y}_\Fs=\ip{x}{f^*y}_\Fs$, 
$x,y\in\widehat{\Ms}$, $f\in \Fs$.
} that satisfy the compatibility condition \hbox{${}_\Fs\ip{x}{y}x=x\ip{y}{x}_\Fs$}, for all $x,y\in\Ms$.\footnote{The compatibility condition assures that the left and right norms induced by the inner products coincide and for bimodule morphisms that are left and right adjointable the left and right adjoints coincide.}

\begin{theorem}
There is an involutive category $\Af^2$ of Hilbert C*-bimodules over unital bimodule C*-algebras factorizable over commutative C*-algebras.
\end{theorem}
\begin{proof}
Objects are unital C*-algebras $\As,\Bs,\Cs,\dots$;
morphisms from $\Bs$ to $\As$ are given by Hilbert \hbox{C*-bimodules} $\Ms$ over unital C*-algebra bimodules factorizable over commutative C*-algebras 
such as $\Rs:=\As\otimes_{C(Y)}\Fs\otimes_{C(X)} \Bs$, $\Ss:=\Bs\otimes_{C(Z)} \Gs\otimes_{C(W)}\Cs$. 

The involution is given by the passage to the contragredient bimodules $\Ms^*$ over $\Bs\otimes_{C(X)}\Fs\otimes_{C(Y)}\As$. 

The composition of $\Ms_\Rs=\As\otimes_{C(Y)}\widehat{\Ms}\otimes_{C(Y)}\Bs$ with 
$\Ns_\Ss=\Bs\otimes_{C(Z)}\otimes\widehat{\Ns}\otimes_{C(W)}\Cs$ is given by the internal tensor product of bimodules 
$\Ms\otimes_\Bs \Ns=\As\otimes_{C(Y)}(\widehat{\Ms}\otimes_{C(X)}\Bs\otimes_{C(Z)}\otimes\widehat{\Ns})\otimes_{C(W)}\Cs$ as a bimodule over $\Rs\otimes_\Bs\Ss\simeq \As\otimes_{C(X)} (\Fs\otimes_{C(Y)}\Bs\otimes_{C(Z)}\Gs)\otimes_{C(W)}\Cs$ with compatible 
$(\Fs\otimes_{C(Y)}\Bs\otimes_{C(Z)}\Gs)$-valued inner products on 
$\widehat{\Ms}\otimes_{C(X)}\Bs\otimes_{C(Z)}\otimes\widehat{\Ns}$ defined by universal factorization property via 
\begin{gather*}
{}_\bullet\ip{x_1\otimes_{C(X)} b_1\otimes_{C(Z)} y_1}{x_2\otimes_{C(X)} b_2\otimes_{C(Z)} y_2}:=
{}_\Fs\ip{x_1}{x_2}\otimes_{C(X)}(b_1b_2^*)\otimes_{C(Z)}{}_\Gs\ip{y_1}{y_2}
\\
\ip{x_1\otimes_{C(X)} b_1\otimes_{C(Z)} y_1}{x_2\otimes_{C(X)} b_2\otimes_{C(Z)} y_2}_\bullet:=
\ip{x_1}{x_2}_\Fs\otimes_{C(X)}(b_1^*b_2)\otimes_{C(Z)}\ip{y_1}{y_2}_\Gs
\end{gather*}
\end{proof}

The previous category can be made into a 2-category $\Af^2$ if we define \hbox{2-arrows} as pairs $(\phi,\Phi)$ such that $\Phi:\Ms_\Rs\to\Ns_\Ss$ is additive map and \hbox{$\phi:\Fs\to\Gs$} is a unital $*$-homomorphism that satisfies $\Phi(r_1xr_2)=\phi(r_1)\Phi(x)\phi(r_2)$, where with some abuse of notation we also denote $1_\As\otimes\phi \otimes 1_\Bs:\Rs\to\Ss$ by $\phi$. Furthermore (at least in the commutative C*-algebras case), one can consider as 2-arrows with source $\Ms_\Rs$ and target $\Ns_\Ss$ new Hilbert C*-bimodules over factorizable \hbox{C*-algebras} bimodules from $\Rs$ to $\Ss$ and in this way the category now constructed becomes actually an $\infty$-category, defining recursively level-$(n+1)$ morphisms as morphisms between the spectral module triples that are morphism at level-$n$.

\section{Section Functor}

\begin{theorem}
There is a section functor $\Gamma:\Ef\to\Af$ that to every propagator $(E,\gamma,R)$ of Hermitian bundles from $(E^1,\pi^1,X^1)$ to $(E^2,\pi^2,X^2)$ associates the Hilbert C*-bimodule $\Gamma(R,E)$ over the \hbox{C*-al}\-gebra bimodule factorizable over commutative C*-algebras 
$C(R)\simeq C(X^1)\otimes_{C(X^1)}C(R)\otimes_{C(X^2)} C(X^2)$.
\end{theorem}
\begin{proof}
The set $\Gamma(R,E)$ of continuous sections of the Hilbert bundle $(E,\gamma,R)$ is already a Hilbert \hbox{C*-bimodule} over the commutative unital 
C*-algebra $C(R)\simeq C(X^2)\otimes_{C(X^2)}C(R)\otimes_{C(X^1)}C(X^1)$ that is a C*-algebra bimodule, factorizable over the commutative C*-algebras $C(X^1)$ and $C(X^2)$. 
\end{proof}

More generally, one can consider propagators where $(E,\gamma,R)$ is a bundle of Hilbert C*-bimodules over a bundle $(A,\gamma',R)$ of commutative C*-algebras (this means that there is a fiber preserving action of the total space $A$ on the total space $E$ making each fiber $E_r$ into a C*-bimodule over the C*algebra $A_r$, for all $r\in R$) and in this way one recovers, via the section functor, a \hbox{C*-bimodule} over the commutative C*-algebra bimodule factorizable over commutative \hbox{C*-algebras} given by 
\hbox{$C(X^1)\otimes_{C(X^1)} \Gamma(R,A)\otimes_{C(X^2)} C(X^2)$}.

\medskip

Let us examine in some more detail how (totally geodesic) maps between compact Riemannian manifolds are described using spectral module triples (this will provide insight on the role of tensorization by B.Mesland bimodules). As already described at the end of the previous section (in the specific case of totally geodesic Riemannian embeddings), every totally geodesic Riemannian map $f:M\to N$ induces a propagator between the complexified Grassmann bundles that is stable under Clifford action and the induced direct sum of the Levi-Civita connections. Perfectly similar results can be formulated for general totally geodesic propagators between Hermitian bundles of Clifford modules with a compatible connection.

Modulo pull-back of bundles and change of rings of modules (that in this commutative situation is not problematic), an application of the section functor
$\Gamma$ will immediately produce a propagator of Hilbert C*-modules over the same C*-algebra $C(f)\simeq C(M)$ and in the totally geodesic case a naive morphism of spectral module triples in $\Sf$.

Alternatively one notes that a propagator between bundles or modules (let's say over the same space) induces at the second quantized level an inclusion into a tensor product factorization.
To explain, in a very special situaton, the tangent bundle decomposition $T(N)|_{f(M)}=Df(T(M))\oplus Df(T(M))^\perp$ corresponds to a  
factorization $\Lambda^\CC(T(N)|_{f(M)})\simeq \Lambda^\CC(Df(T(M)))\otimes\Lambda^\CC(Df(T(M))^\perp)$ of 
Grassmann bundles  and so to a tensorial factorization of the bimodules of sections. In this way we see a possible role for $\Gamma(\Lambda^\CC(Df(T(M))^\perp))$ as a Mesland bimodule for the Hodge-De Rham spectral triples of $M$ and $N$. We plan to elaborate much further on these points in forthcoming work.

\section{Outlook}

The work here presented is at a very preliminary stage and most of the elementary categorical structures here considered are essentially a playground (still mainly at the topological level) to test the validity of some conjectures. Specifically we would like to see a clear picture of how geometrical morphisms of Riemannian manifolds can be encoded via the section functor in terms of B.Mesland's bimodules between commutative Hodge-De Rham spectral triples. In order to provide a duality, a spectrum functor from categories of commutative Riemannian spectral triples to Riemannian manifolds must be constructed. At the level of objects this is already done, via the already mentioned reconstruction theorems by A.Connes and A.Rennie, S.Lord, J.Varilly, and our next goal is to prove a similar reconstruction theorem for suitable (totally geodesic) morphisms between these Hodge-De Rham spectral triples. Our hope is that, if morphisms can be described as a bivariant version of spectral triples, a direct application of (part of) the reconstruction theorems for objects might be possible also in the case of morphisms.

Another important direction of investigation is related to our belief that ``involutive tensorial'' categories are the right environment for
the study of non-commutative geometry and that involutive categories of bimodules should help to formulate a version of B.Mesland category of ``bivariant'' spectral triples with involutions. The categories of Hilbert \hbox{C*-bimodules} over \hbox{C*-algebra} bimodules factorizable over commutative C*-algebras that we defined here are not yet sufficient to cover even some of the most elementary morphisms of non-commutative spaces (the bimodule $\Bs_\As$ induced by a unital $*$-homomorphism $\phi:\As\to\Bs$, for example).\footnote{
A more satisfactory treatment of morphisms of non-commutative spaces (even at the topological level) is well-beyond the scope of such elementary paper and will likely require the usage of higher-C*-categories.  
}


\newpage 

\emph{Notes and Acknowledgments:} 
The first author thanks his long time collaborator R.Conti at the ``Sapienza'' University in Rome for the discussion of many topics related to this research. He also thanks Starbucks Coffee at the $1^{\text{st}}$ floor of Emporium Tower in Sukhumvit, where he spent most of the time dedicated to this research project. 

\smallskip 

\textit{We stress that the two authors do not share any of their ideological, religious, political affiliations}.

\end{document}